\documentclass[conference]{IEEEtran}
\IEEEoverridecommandlockouts
\usepackage{cite}
\usepackage{amsmath,amssymb,amsfonts}
\usepackage{mathrsfs}
\usepackage{algorithmic}
\usepackage{graphicx}
\usepackage{textcomp}
\usepackage[colorlinks=true, allcolors=blue]{hyperref}
\usepackage{color}

\newtheorem{remark}{Remark}
\newtheorem{definition}{Definition}
\newtheorem{proposition}{Proposition}
\newtheorem{theorem}{Theorem}
\newtheorem{lemma}{Lemma}

\newtheorem{assumption}{Assumption}

\def\BibTeX{{\rm B\kern-.05em{\sc i\kern-.025em b}\kern-.08em
    T\kern-.1667em\lower.7ex\hbox{E}\kern-.125emX}}
\begin{document}

\title{\huge{Automated Synthesis of Lyapunov Functions for Multi-Agent Systems under Jointly Connected Topology}
\thanks{*Corresponding author. This work was supported by the National Natural Science Foundation of China under Grant 62573014.}
}

\author{\IEEEauthorblockN{Shuyuan Zhang}
\IEEEauthorblockA{\textit{ICTEAM Institute} \\
\textit{UCLouvain}\\
Louvain-la-Neuve, Belgium \\
shuyuan.zhang@uclouvain.be}
\and
\IEEEauthorblockN{Lei Wang*}
\IEEEauthorblockA{\textit{Sch. Automat. Sci. Elect. Engineering} \\
\textit{Beihang University}\\
Beijing, China \\
lwang@buaa.edu.cn}
\and
\IEEEauthorblockN{Qing-Guo Wang}
\IEEEauthorblockA{\textit{Sch. Mechanical Engineering} \\
\textit{Wuhan University of Science and Technology}\\
Wuhan, China \\
wangqgsg@wust.edu.cn}
}

\maketitle

\begin{abstract}
This article investigates the consensus tracking problem of multi-agent systems under jointly connected topology through automated synthesis of Lyapunov functions. Based on the proposed distributed nonlinear control protocol, several consensus criteria for first-order multi-agent systems are established and certified by the construction and synthesis of more general polynomial Lyapunov functions. By employing sum-of-squares decomposition for multivariate polynomials, we can efficiently synthesize polynomial Lyapunov functions to achieve consensus verification in polynomial time, although the widely used quadratic Lyapunov functions do not exist. Moreover, polynomial coupling functions for our proposed protocol are concomitantly generated. Furthermore, the distributed nonlinear control protocol is extended to deal with second-order multi-agent systems, while ensuring second-order consensus verification. 
Finally, an example is presented to demonstrate the efficacy of our method.
\end{abstract}

\begin{IEEEkeywords}
Automated synthesis, polynomial Lyapunov functions, consensus verification, multi-agent systems, sum-of-squares.
\end{IEEEkeywords}

\section{Introduction}
\label{1}

In recent years, multi-agent systems (MASs) have attracted extensive research attention owing to their broad applicability in areas such as satellite formation flying, cooperative control of multiple vehicles, and various distributed coordination tasks.
Within the systems and control community, one of the fundamental research topics of MASs is consensus \cite{Egerstedt:2007,Andreasson:2014,Jiang:2020,Valcher:2017}, which is to drive all the agents to realize common dynamical behaviors via exchanging information among adjacent agents.

Research on consensus of MASs mainly focuses on identifying connectivity conditions for network topology, to ensure that the agents in MASs can achieve consensus.
The undirected and directed graphs, and the fixed and switching topologies have been widely investigated in the literature (e.g., \cite{Wang:2023,Stilwell:2006,Qin:2020,Jadbabaie:2003,Ren:2005,Cheng:2019,Hong:2007,Olfati-Saber:2006,Su:2009,Su:2019}),
where the joint connectivity is the mildest connectivity assumption on the network topology. In \cite{Jadbabaie:2003,Ren:2005,Cheng:2019}, linear consensus protocols were proposed to study MASs with single integrator dynamics and linear dynamics, to guarantee the group of agents to achieve consensus under the jointly connected switching topology.
Meanwhile, there is a growing interest in consensus algorithms that the agents are governed by second-order dynamics, which is extremely meaningful for the implementation of cooperative control strategies in engineering networked systems. Olfati-Saber in \cite{Olfati-Saber:2006} investigated the leader-follower tracking problem of second-order MASs, and later Su \emph{et al.} in \cite{Su:2009} extended it to the case that only a subset of agents have access to the leader's information.
In \cite{Su:2019}, Su and Lin studied the coordinated tracking control under a jointly connected topology, both in the presence and in the absence of the virtual leader's velocity information.

However, the studies mentioned above primarily concentrate on constructing quadratic Lyapunov functions to demonstrate the first-order and second-order consensus of MASs.
In addition to the challenge of acquiring quadratic Lyapunov functions, another problem being faced is that the existence of quadratic Lyapunov functions cannot be guaranteed.
Therefore, more complex candidates of Lyapunov functions may be required for certain systems. 
It is desirable to be able to determine Lyapunov functions algorithmically.
Our aim is to address the consensus problem of MASs under the jointly connected topology by automatically synthesizing more general Lyapunov functions instead of the quadratic form.
In addition, compared with the existing distributed control protocols, such as the traditional linear control protocols \cite{Jadbabaie:2003,Ren:2005,Cheng:2019,Su:2019}, and the nonlinear control protocols with the complicated gradient term \cite{Olfati-Saber:2006,Su:2009}, a more comprehensive distributed control protocol need to be provided in a convenient way.

Gratefully, the sum-of-squares (SOS) decomposition framework provides an effective way to compute polynomial Lyapunov functions on basis of the semi-definite programming \cite{Papachristodoulou:2002,Papachristodoulou:20052,Chesi:2013,Zhang:2022}.
In \cite{Papachristodoulou:2002}, Papachristodoulou and Prajna proposed an idea of algorithmic construction of polynomial Lyapunov functions to study the stability for nonlinear systems. And the tutorial was developed for system analysis based on the SOS decomposition in \cite{Papachristodoulou:20052}.
Later, the authors in \cite{Chesi:2013,Zhang:2022} utilized the method based on the SOS decomposition to certify the synchronization criteria for polynomial networked systems by the algorithmic search of polynomial Lyapunov functions.
In \cite{zhang2024consensus}, polynomial coupling functions were automatically synthesized using SOS programming to design a distributed control protocol that ensures consensus for MASs.
This work is an extension of our previous studies in \cite{Zhang:2022,zhang2024consensus}, from fixed topology to switching topology.

In this article, we investigate the consensus tracking problem of MASs under jointly connected topology using the SOS programming technique.
First, a distributed nonlinear control protocol is proposed for first-order MASs, which is a natural extension of the common linear control protocols \cite{Jadbabaie:2003,Ren:2005,Cheng:2019}.
Different from the existing methods which are dependent on constructing the quadratic Lyapunov functions,
the consensus criteria for MASs under the jointly connected topology are established with more general polynomial Lyapunov functions.
Then, based on the SOS decomposition for multivariate polynomials, we can efficiently synthesize the polynomial Lyapunov functions to verify consensus in polynomial time, and also the polynomial coupling functions for our proposed protocol. 
This approach remains effective even when quadratic Lyapunov functions are not applicable.
Furthermore, the distributed nonlinear control protocol is extended to handle second-order MASs, while ensuring automatic verification of second-order consensus. 
At last, a numerical example is presented to demonstrate the performance of our approach.

The remaining sections of this article are structured as follows. Section \ref{2} formulates the problem of interest, and some notations, definitions and lemmas are used throughout this article. In Section \ref{3}, the consensus criteria for MASs are established and certified by the automated synthesis of polynomial Lyapunov functions. After presenting an example in Section \ref{sim}, we conclude this article in Section \ref{4}.

\emph{Notations}: Throughout this article, $\mathbb{R}_{+}$ and $\mathbb{R}$ denote the sets of positive real numbers and real numbers, respectively. $\mathbb{R}^{n}$ is the set of real vectors of $n$-dimension, $\mathbb{R}^{m \times m}$ is the set of $m\times m$ dimensional real matrices, $\mathbb{R}[\cdot]$ is the set of real polynomials and $\mathbb{R}_{\tau}[\cdot]$ is the set of real polynomials of degree at most $\tau$. $I_{M}$ is an $M \times M$ dimensional identity matrix, $1_{M}$ is a column vector of $M$-dimension with each entry being $1$, $\text{diag}\{\nu_{1},\nu_{2},\ldots,\nu_{n}\}$ is a diagonal matrix with entries $\nu_{1},\nu_{2},\ldots,\nu_{n}$. $|\cdot|$ is the absolute value symbol, and $||\cdot||$ is any vector norm. For matrices $\Phi$ and $\Psi$, the Kronecker product is expressed by $\Phi\otimes \Psi$. $\Phi^{\top}$ represents the transpose of $\Phi$, and $\Psi>0$ means that $\Psi$ is positive definite. For a continuously differentiable function $V(\psi):\mathbb{R}^n\rightarrow\mathbb{R}$, let $\nabla V(\psi) = [\frac{\partial V(\psi)}{\partial \psi}]^{\top}$ and $\nabla V(\phi-\varphi) = [\frac{\partial V(\psi)}{\partial \psi}]^{\top}|_{\psi=\phi-\varphi}$, where $\psi, \phi, \varphi \in \mathbb{R}^n$. $\text{SOS}[z]$ denotes the set of multivariate SOS polynomials, i.e.,
\[\text{SOS}[z] = \left \{\rho(z) \in \mathbb{R}[z]  \middle |\;
\begin{aligned}
 &\rho(z)=\sum_{i=1}^k p_i^2(z), \forall z \in \mathbb{R}^{n} \\
 &p_i(z)\in \mathbb{R}[z], i=1,\ldots,k
\end{aligned}
\right\}.
\]

\section{Preliminaries}
\label{2}

Graph theory is a vital tool to depict the network topology, which is introduced as follows. Let $\mathscr{G}=(\mathscr{V},\mathscr{E},A)$ be a undirected graph with $N$ nodes, where $\mathscr{V}=\{v_{1},v_{2},\ldots,v_{N}\}$ is the set of nodes and $\mathscr{E}\subseteq\mathscr{V}\times\mathscr{V}$ is the set of edges between nodes. The weights of edges $\{v_{j},v_{k}\}$ are denoted by $A_{jk}\geq0$, and matrix $A=(A_{jk})\in \mathbb{R}^{N\times N}$ represents the weighted adjacent matrix, which implies the information exchange between node $j$ and node $k$. Define the Laplacian matrix $L=(L_{jk})\in \mathbb{R}^{N \times N}$ as $L_{jk}=-A_{jk}$ for $j\neq k$ and $L_{jj}=\sum_{k=1,k\neq j}^N A_{jk}$. A path of length $l$ is a sequence of edges $(v_{j_{1}},v_{j_{2}})$, $(v_{{j}_{2}},v_{j_{3}})$, $\cdots$, $(v_{j_{l}},v_{j_{l+1}})$ with $l+1$ distinct nodes. Graph $\mathscr{G}$ is said to be connected if there exists a path connecting any two distinct nodes.

\begin{definition} \cite{Shi:2009} The union of a collection of graphs $\mathscr{G}_{1},\mathscr{G}_{2},\ldots,\mathscr{G}_{m}$ is defined as $\mathscr{G}_{M}$, where the node and edge sets of $\mathscr{G}_{M}$ comprise the union of the node and edge sets of all graphs in the collection. $\mathscr{G}_{1},\mathscr{G}_{2},\ldots,\mathscr{G}_{m}$ is said to be jointly connected if the union graph $\mathscr{G}_{M}$ is a connected graph.
\end{definition}

Consider an MAS with $N$ follower agents and one virtual leader, in which the dynamics of each follower agent $i$ are described by
\begin{align}
\dot{z}_{i}(t)= u_{i}(t), \label{MAS}
\end{align}
with $i=1,\ldots,N$, where $z_{i}(t)=[z_{i1}(t),\ldots,z_{in}(t)]^{\top}\in \mathbb{R}^{n}$ is the state vector of agent $i$, and $u_{i}(t) \in \mathbb{R}^{n}$ is the control input for the $i$-th agent, which is expressed by the following distributed nonlinear control protocol:
\begin{align}
u_{i}=\sum_{j \in \mathcal{N}_{i}}A_{ij}^{\sigma}(h({z}_{j})-h({z}_{i})) +d_{i}^{\sigma}(h({z}_{\gamma})-h({z}_{i})), \label{Protocol}
\end{align}
where $\mathcal{N}_{i}$ is the neighborhood set of the $i$-th agent;
$A_{ij}^{\sigma}$ is the adjacent matrix associated with the switching graph $\mathscr{G}_{\sigma(t)}$, and $\sigma(t):[0,\infty)\rightarrow \mathcal{M}=\{1,\ldots,m\}$ is a switching signal, which is a piecewise constant function; $d_{i}^{\sigma}>0$ if agent $i$ is a neighbor of the leader and $d_{i}^{\sigma}=0$ otherwise; $h(\cdot)=[h_{1},\ldots,h_{n}]^{\top}\in \mathbb{R}^{n} \rightarrow \mathbb{R}^{n}$ is the polynomial coupling function;
${z}_{\gamma}$ is a constant vector generated by the virtual leader as the reference.

Based on the distributed nonlinear control protocol \eqref{Protocol}, our goal is to drive the state trajectories of all the follower agents to achieve the consensus tracking of the virtual leader. The definition of consensus is presented exactly as follows.

\begin{definition} \cite{Yu:2011} System \eqref{MAS} achieves consensus if for $\forall \epsilon >0$, there exists $T \in \mathbb{R}_{+}$ such that
\begin{align}
||z_{i}(t)-{z}_{\gamma}|| \leq \epsilon, \ \forall t\geq T  \label{D1}
\end{align}
with $i=1,\ldots,N$ for any initial conditions.
\end{definition}

For the sake of convenience, we divide the time interval $[0,T)$ into a finite sequence of bounded, non-overlapping and contiguous time intervals $[t_{\kappa}, t_{\kappa+1})$ for $\kappa=0,1,2,\ldots$ with $t_{0}=0$, and $t_{\kappa+1}>t_{\kappa}$. For each
$[t_{\kappa}, t_{\kappa+1})$, suppose that there exists a sequence of non-overlapping,
contiguous subintervals $[t_{\kappa}^{0},t_{\kappa}^{1}),\ldots,[t_{\kappa}^{l},t_{\kappa}^{l+1}),\ldots,
[t_{\kappa}^{s_{\kappa}-1},t_{\kappa}^{s_{\kappa}})$
with $t_{\kappa}= t_{\kappa}^{0}$, and $t_{\kappa+1}= t_{\kappa}^{s_{\kappa}}$ for some positive integers $s_{\kappa}$.
In this article, we suppose that there exists a dwell time $\tau \geq 0$ satisfying $t_{\kappa}^{l+1}-t_{\kappa}^{l} \geq \tau$ for $l= 0,1,\ldots,s_{\kappa}-1$. Note that the switching graph only switches at time instants $t_{\kappa}^{0},t_{\kappa}^{1},\ldots, t_{\kappa}^{s_{\kappa}-1}$, that is, the switching topology is fixed in each time subinterval $[t_{\kappa}^{l},t_{\kappa}^{l+1})$. Here, each graph on the time subinterval may be disconnected, while we only require that the union of graphs $\mathscr{G}_{\sigma(t)}$ on the time interval $[t_{\kappa}, t_{\kappa+1})$ is jointly connected.

\begin{lemma}
Denote $D=\text{diag}\{d_{1},d_{2},\ldots,d_{N}\}$. Let the Laplacian matrices $L_{1},L_{2},\ldots,L_{m}$ and the diagonal matrices $D_{1},D_{2},\ldots,D_{m}$ be associated with graphs $\mathscr{G}_{1},\mathscr{G}_{2},\ldots,\mathscr{G}_{m}$, we have $\sum_{i=1}^{m} (L_{i}+D_{i})$ is positive definite if these graphs are jointly connected.
\end{lemma}

\begin{lemma} \cite{Chen:2008}
Given a symmetric Laplacian matrix $L=(L_{ij}) \in \mathbb{R}^{N \times N}$, it holds that
\begin{align}
&\sum_{i=1}^{N}\sum_{j=1}^{N} \alpha_{i}^{\top}L_{ij}\beta_{j} \nonumber\\
&= -\sum_{i=1}^{N-1}\sum_{j=i+1}^{N}(\alpha_{i}-\alpha_{j})^{\top}L_{ij}(\beta_{i}-\beta_{j}), \label{L2}
\end{align}
where $\alpha_{i}$, $\beta_{j} \in \mathbb{R}^{n}$ for $i,j=1,\ldots,N$.
\end{lemma}

\begin{lemma} \cite{Yu:2014}
For a uniformly continuous and nonnegative function $g(t)$, if
\begin{align}
\lim_{t \rightarrow \infty} \int_{t}^{t+\epsilon} g(\chi) d\chi =0
\label{L3}
\end{align}
holds for some $\epsilon>0$, then it follows that $\lim_{t \rightarrow \infty} g(t)=0$.
\end{lemma}

\section{Main Results}
\label{3}

In this section, we discuss the problems of consensus analysis and verification of MASs with the joint connectivity condition. Several sufficient criteria are proposed by the construction and synthesis of polynomial Lyapunov functions to ensure the consensus of first-order MASs. Moreover, we extend the proposed method to investigate MASs with second-order dynamics under the jointly connected topology.

\subsection{Consensus Verification for First-order MASs}
\label{3-1}

In this subsection, the consensus criteria are established for MAS \eqref{MAS} based on the distributed nonlinear control protocol \eqref{Protocol}. The polynomial Lyapunov functions are automatically synthesized to achieve consensus verification via SOS programming.

To begin with, the following assumption about $h(\cdot)$ is given.

\begin{assumption} For $h(\cdot)\in \mathbb{R}[\cdot]: \mathbb{R}^n\rightarrow\mathbb{R}^n$, there exist a continuously differentiable, radially unbounded, positive definite polynomial function $V(\cdot):\mathbb{R}^n\rightarrow\mathbb{R}_{+}$, a vector valued function of polynomial $q(\cdot):\mathbb{R}^n\rightarrow\mathbb{R}^m$ with $q(\alpha)=0 \Leftrightarrow \alpha = 0$ such that for $\forall \alpha, \beta, \gamma \in \mathbb{R}^n$, there holds
\begin{align}
F_{\nabla V}(\alpha-\gamma,\beta-\gamma)F_{h}(\alpha,\beta) \geq  H_{q,\Psi}(\alpha-\gamma,\beta-\gamma), \label{A1}
\end{align}
where $F_{\nabla V}(\alpha,\beta)= \nabla V(\alpha)-\nabla V(\beta)$, $F_{h}(\alpha,\beta)= h(\alpha)-h(\beta)$, $H_{q,\Psi}(\alpha,\beta)=[q(\alpha)-q(\beta)]^{\top}\Psi[q(\alpha)-q(\beta)]$, and $\Psi \in \mathbb{R}^{m\times m}$ is a positive definite matrix.
\end{assumption}

\begin{remark}
It should be noted that if $V(\cdot)$ is a function of quadratic form, i.e., $V(z)=\frac{1}{2}z^{\top}Pz$, where $P$ is a positive definite matrix, then condition \eqref{A1} turns out to be $(\alpha-\beta)^{\top}P(h(\alpha)-h(\beta)) \geq 0$, which is a frequently used constraint
about $h(\cdot)$ \cite{Saber:2003,Liang:2020}. More specifically, if $q(z)=z$, the inequality \eqref{A1} reduces to $(\alpha-\beta)^{\top}P(h(\alpha)-h(\beta))\geq (\alpha-\beta)^{\top} \Psi(\alpha-\beta)$, which is usually used in quadratic Lyapunov function method to study the consensus problem (e.g., \cite{Chen:2008,Wei:2020,Rossa:2023}). Thus, the constraint on $h(\cdot)$ is reasonable. 
Moreover, Assumption 1 has been adopted in many practical systems; see, for example, the examples reported in \cite{Zhang:2022, xiao2026synchronization}.
\end{remark}

\begin{theorem}
Assuming that Assumption 1 is valid. If the switching topology graph among the follower and leader agents is jointly connected across each time interval $[t_{\kappa},t_{\kappa+1})$ for $\kappa=0,1,\ldots$, then the consensus of system \eqref{MAS} is achieved under protocol \eqref{Protocol}.
\end{theorem}

\emph{Proof:} Let $\tilde{V}(t)=\sum_{i=1}^{N}V(z_{i}-z_{\gamma})$, where $V$ satisfies Assumption 1. Taking the derivative of $\tilde{V}(t)$ gives
\begin{align}
\dot{\tilde{V}}(t) &= \sum_{i=1}^{N} \sum_{j \in \mathcal{N}_{i}} \nabla V(z_{i}-z_{\gamma})
                   A_{ij}^{\sigma}(h({z}_{j})-h({z}_{i})) \nonumber\\
                &\ \ \ \ +\sum_{i=1}^{N}d_{i}^{\sigma} \nabla V(z_{i}-z_{\gamma})
                   (h({z}_{\gamma})-h({z}_{i}))  \nonumber\\
                &= -\sum_{i=1}^{N} \sum_{j =1}^{N} \nabla V(z_{i}-z_{\gamma})
                   L_{ij}^{\sigma}h({z}_{j}) \nonumber\\
                &\ \ \ \ -\sum_{i=1}^{N}d_{i}^{\sigma} \nabla V(z_{i}-z_{\gamma})
                   (h({z}_{i})-h({z}_{\gamma}))  \nonumber\\
                &= \sum_{i=1}^{N-1}\sum_{j=i+1}^{N} F_{\nabla V}(z_{i}-z_{\gamma},z_{j}-z_{\gamma}) L_{ij}^{\sigma} F_{h}(z_{i},z_{j}) \nonumber\\
                &\ \ \ \ -\sum_{i=1}^{N}d_{i}^{\sigma} \nabla V(z_{i}-z_{\gamma})
                   F_{h}(z_{i},z_{\gamma}), \label{TH1-1}
\end{align}
where Lemma 2 has been used. It is seen from \eqref{A1} that
\begin{align}
\dot{\tilde{V}}(t) & \leq \sum_{i=1}^{N-1}\sum_{j=i+1}^{N} L_{ij}^{\sigma}
                  H_{q,\Psi}(z_{i}-z_{\gamma},z_{j}-z_{\gamma})  \nonumber\\
                &\ \ \ \ -\sum_{i=1}^{N} d_{i}^{\sigma} q^{\top}(z_{i}-z_{\gamma}) \Psi q(z_{i}-z_{\gamma}) \nonumber\\
                &= - \sum_{i=1}^{N} \sum_{j =1}^{N} L_{ij}^{\sigma} q^{\top}(z_{i}-z_{\gamma}) \Psi q(z_{j}-z_{\gamma}) \nonumber\\
                &\ \ \ \ -\sum_{i=1}^{N} d_{i}^{\sigma} q^{\top}(z_{i}-z_{\gamma}) \Psi q(z_{i}-z_{\gamma}) \nonumber\\
                &= -Q^{\top}(z-z_{\gamma})(G^{\sigma} \otimes \Psi)Q(z-z_{\gamma}), \label{TH1-2}
\end{align}
where $Q(z-z_{\gamma})=[q^{\top}(z_{1}-z_{\gamma}),\ldots,q^{\top}(z_{N}-z_{\gamma})]^{\top}$, and $G^{\sigma}= L^{\sigma}+D^{\sigma}$. Since $G^{\sigma}$ is positive semi-definite, it follows that $\dot{\tilde{V}}(t) \leq 0$. Thus, $\tilde{V}(t)$ is a non-increasing function and has a lower bound. According to Cauchy's convergence criteria, for any $\varepsilon > 0$, there exists $N_{\varepsilon}>0$ such that
\begin{align}
|\tilde{V}(t_{\kappa+1})-\tilde{V}(t_{\kappa})| < \varepsilon, \quad \forall \kappa \geq N_{\varepsilon}  \label{TH1-3}
\end{align}
i.e.,
\begin{align}
|\int_{t_{\kappa}}^{t_{\kappa+1}}\dot{\tilde{V}}(\varpi)d\varpi| < \varepsilon.  \label{TH1-4}
\end{align}
Then, we have
\begin{align}
\int_{t_{\kappa}^{0}}^{t_{\kappa}^{1}}\dot{\tilde{V}}(\varpi)d\varpi +\ldots +\int_{t_{\kappa}^{s_{\kappa}-1}}^{t_{\kappa}^{s_{\kappa}}}\dot{\tilde{V}}(\varpi)d\varpi
> -\varepsilon.  \label{TH1-5}
\end{align}
Since $s_{\kappa}$ is finite, it follows that
\begin{align}
-\varepsilon &< \int_{t_{\kappa}^{i}}^{t_{\kappa}^{i+1}}\dot{\tilde{V}}(\varpi)d\varpi  \nonumber\\
&< -\int_{t_{\kappa}^{i}}^{t_{\kappa}^{i+1}} Q^{\top}(\omega(\varpi))(G^{\sigma(t_{\kappa}^{i})} \otimes \Psi)Q(\omega(\varpi)) d\varpi \nonumber\\
&< -\int_{t_{\kappa}^{i}}^{t_{\kappa}^{i}+\tau} Q^{\top}(\omega(\varpi))(G^{\sigma(t_{\kappa}^{i})} \otimes \Psi)Q(\omega(\varpi)) d\varpi \label{TH1-6}
\end{align}
with $i=0,1,\ldots,s_{\kappa-1}$, where $\omega=[\omega_{1}^{\top},\ldots,\omega_{N}^{\top}]^{\top}$ and $\omega_{i}=z_{i}-z_{\gamma}$ for $i=1,\ldots,N$. Thus, we have
\begin{align}
\int_{t_{\kappa}^{i}}^{t_{\kappa}^{i}+\tau} Q^{\top}(\omega(\varpi))(G^{\sigma(t_{\kappa}^{i})} \otimes \Psi)Q(\omega(\varpi)) d\varpi < \varepsilon. \label{TH1-7}
\end{align}
Hence,
\begin{align}
&\lim_{t\rightarrow \infty} \int_{t}^{t+\tau} Q^{\top}(\omega(\varpi))
(G^{\sigma(t_{\kappa}^{0})}\otimes \Psi + \ldots \nonumber\\
&\quad \quad \quad \quad \quad \quad \quad \quad \quad + G^{\sigma(t_{\kappa}^{s_{\kappa}-1})}\otimes \Psi)Q(\omega(\varpi)) d\varpi \nonumber\\
&=\lim_{t\rightarrow \infty} \int_{t}^{t+\tau} Q^{\top}(\omega(\varpi))
(\sum_{j=0}^{s_{\kappa}-1} G^{\sigma(t_{\kappa}^{j})}\otimes \Psi)Q(\omega(\varpi)) d\varpi \nonumber\\
&=0. \label{TH1-8}
\end{align}
It follows from $\dot{\tilde{V}}(t) \leq 0$ that $\omega(t)$ is bounded so is its derivative $\dot{\omega}(t)$. Thus, we know that $Q^{\top}(\omega(t))
(\sum_{j=0}^{s_{\kappa}-1} G^{\sigma(t_{\kappa}^{j})}$ $\otimes \Psi)Q(\omega(t))$ is uniformly continuous. By Lemma 3, we have $\lim_{t\rightarrow \infty} Q^{\top}(\omega(t))(\sum_{j=0}^{s_{\kappa}-1} G^{\sigma(t_{\kappa}^{j})} \otimes \Psi)Q(\omega(t)) = 0$. Based on Lemma 1, $\sum_{j=0}^{s_{\kappa}-1} G^{\sigma(t_{\kappa}^{j})}$ is positive define because of the joint connectivity of the interconnected graphs during $[t_{\kappa},t_{\kappa+1})$, then we can
conclude that $\lim_{t\rightarrow \infty} Q(\omega(t))=0$, i.e.,
$\omega_{i}(t)\rightarrow 0$ for $i=1,\ldots,N$ as $t\rightarrow \infty$, which indicates that the consensus of system \eqref{MAS} under the jointly connected topology is achieved. The proof is thus completed.  \hfill{$\square$}

On basis of protocol \eqref{Protocol}, we establish the consensus criterion for MAS \eqref{MAS} subjected to the jointly connected topology by constructing polynomial Lyapunov functions. It is noteworthy to mention that we presume the existence of function $V(\cdot)$ in Assumption 1. In the following, we will present our method for automated synthesis of such a function. However, there are few efficient computational methods to synthesize $V(\cdot)$ involving multivariate non-negativity conditions. To the best of our knowledge, even checking the non-negativity of a quartic polynomial is NP-hard \cite{Papachristodoulou:2005}.

The SOS programming provides convex relaxations for a variety of computationally difficult optimization problems, which has been a powerful method to synthesize Lyapunov functions in control theory and engineering (e.g., \cite{Papachristodoulou:2002,Chesi:2013,Zhang:2022,zhang2023stability,zhang2025automatic}). Thus, we can use the SOS relaxations as an alternative for certifying non-negativity because of that an SOS polynomial is completely non-negative.
That is, the non-negativity condition in Assumption 1 can be characterized by the SOS condition. In this way, we can automatically synthesize the polynomial Lyapunov functions to achieve consensus verification via SOS programming, for which the powerful algorithm and off-the-shelf software are available.

\begin{assumption} For $h(\cdot)\in \mathbb{R}[\cdot]:\mathbb{R}^n\rightarrow\mathbb{R}^n$, there exist an SOS polynomial function $V(\cdot):\mathbb{R}^n\rightarrow \mathbb{R}_{+}$, a vector valued function of polynomial $q(\cdot):\mathbb{R}^n\rightarrow\mathbb{R}^m$ with $q(\alpha)=0 \Leftrightarrow \alpha = 0$ such that for $\forall \alpha, \beta, \gamma \in \mathbb{R}^n$, there holds
\begin{align}
&q^{\top}(\alpha) \Psi q(\eta)- \alpha^{2m} \in {\rm SOS}[\alpha], \label{A2-1} \\
&F_{\nabla V}(\alpha-\gamma,\beta-\gamma)F_{h}(\alpha-\beta) \nonumber\\
&\quad \quad -H_{q,\Psi}(\alpha-\gamma,\beta-\gamma)\in {\rm SOS}[\alpha,\beta,\gamma]. \label{A2-2}
\end{align}
\end{assumption}

\begin{proposition}
Assuming that Assumption 2 is valid. If the switching topology graph among the follower and leader agents is jointly connected across each time interval $[t_{\kappa},t_{\kappa+1})$ for $\kappa=0,1,\ldots$, then the consensus of system \eqref{MAS} is achieved under protocol \eqref{Protocol}.
\end{proposition}

\begin{remark}
The SOS condition \eqref{A2-1} guarantees $\Psi$ is positive definite. Note that $q(\cdot)$ and $\Psi$ are both unknown in \eqref{A2-1} and \eqref{A2-2}. So we presume that $\hat{q}(\cdot)$ is a vector valued
function of monomial, and let $q^{\top}\Psi q=\hat{q}^{\top}\hat{\Psi}\hat{q}$.
Then, we only need to solve the positive definite matrix $\hat{\Psi}$ by SOS programming.
Moreover, by constructing $\hat{q}(\alpha)$ to include the linear monomials of $\alpha$, one can readily ensure that $\hat{q}(\alpha) = 0 \Leftrightarrow \alpha = 0$.
\end{remark}

\begin{remark}
Although \eqref{A2-2} still relates to a bilinear problem and thus is NP-hard \cite{Toker:1995}, it can be solved by existing solvers such as PENBMI \cite{Henrion:2005}. Besides, in practice, engineering insights can facilitate the construction of the polynomial coupling function $h(\cdot)$. We can empirically provide initial parameter values of $h(\cdot)$ in advance to avoid the bilinear problem, which eventually falls within the framework of convex optimization. It is known for being computationally tractable, which means that the optimization problem can be addressed efficiently in polynomial time.

In this article, we use YALMIP toolbox \cite{Johan:2004} and MOSEK solver \cite{aps2019mosek} to synthesize the desired polynomial Lyapunov function $V(\cdot)$, which is an efficient optimization software to solve the aforementioned SOS programming problem.
\end{remark}

\subsection{Consensus Verification for Second-order MASs}
\label{3-2}
In subsection \ref{3-1}, the distributed nonlinear control protocol \eqref{Protocol} with the position states $z_{i}$ is designed to discuss the first-order consensus problem of MAS \eqref{MAS} with single integrator dynamics. However, in most of the applications of MASs, such as UAVs formation \cite{Liu:2018,Gruszka:2013}, the reference trajectory $z_{\gamma}$ is time-varying, and thus the information of velocity states is critical. For this purpose, we will extend the SOS programming approach to investigate the second-order MASs with double integrator dynamics in this subsection.

We consider that only some, but not all, of the follower agents are received the position and velocity information of the virtual leader. Then, the distributed nonlinear control protocol is designed as follows:
\begin{align}
u_{i}=&\sum_{j \in \mathcal{N}_{i}}A_{ij}^{\sigma}h_{1}({z}_{j}-{z}_{i}) + \sum_{j \in \mathcal{N}_{i}}A_{ij}^{\sigma}(h_{2}({v}_{j})-h_{2}({v}_{i})) \nonumber\\ &+d_{i}^{\sigma}h_{1}({z}_{\gamma}-{z}_{i})
+d_{i}^{\sigma}(h_{2}({v}_{\gamma})-h_{2}({v}_{i}))
\label{Protocol2}
\end{align}
with $i=1,\ldots,N$, where $z_{i} \in \mathbb{R}^{n}$ and $v_{i} \in \mathbb{R}^{n}$ are the position and velocity states of agent $i$; $h_{1}(\cdot)\in \mathbb{R}^{n} \rightarrow \mathbb{R}^{n}$ and $h_{2}(\cdot)\in \mathbb{R}^{n} \rightarrow \mathbb{R}^{n}$ are both polynomial coupling functions, which satisfies $h_{1}(-a)=-h_{1}(a)$ for $\forall a \in \mathbb{R}^{n}$; ${z}_{\gamma}$ and ${v}_{\gamma}$ are the reference states of position and velocity generated by the virtual leader.

In protocol \eqref{Protocol2}, two different types of nonlinear couplings are considered according to the way of information acquirement. (i) The MASs are equipped with active sensors such as radar and sonar that can directly detect the relative states. Mathematically, we use the coupling function $h_{1}(\cdot)$ to describe the information of relative state couplings for position. (ii) Neighboring agents may also send their absolute information via broadcast. Thus, we use the coupling function $h_{2}(\cdot)$ to describe the information of absolute state couplings for velocity.
In both cases, we only use the relative information exchange among the neighboring agents, and thus protocol \eqref{Protocol2} is said to be the distributed nonlinear control protocol. Note that the designed protocol \eqref{Protocol2} is easier to compute than those in \cite{Olfati-Saber:2006} and \cite{Su:2009}, where a nonlinear gradient term derived from a complicated artificial potential function appears.
Moreover, if $h_{1}(\cdot)$ and $h_{2}(\cdot)$ are linear coupling functions, then protocol \eqref{Protocol2} becomes the classical linear control law in most literature (e.g., \cite{Su:2019,Qin:2016,Difilippo:2021}).

Defining the state errors of position and velocity as $\omega_{i}=z_{i}-z_{\gamma}$ and $\nu_{i}=v_{i}-v_{\gamma}$ for $i=1,\ldots,N$, the error dynamics of each agent in second-order MASs are defined as 
\begin{align}
\left\{\begin{aligned}
&\dot{\omega}_{i} =  h_{1}(\nu_{i}) \\
&\dot{\nu}_{i}     =\sum_{j \in \mathcal{N}_{i}}A_{ij}^{\sigma}h_{1}({\omega}_{j}-{\omega}_{i}) +d_{i}^{\sigma}h_{1}({\omega}_{\gamma}-{\omega}_{i}) \\
&+ \sum_{j \in \mathcal{N}_{i}}A_{ij}^{\sigma}(h_{2}({v}_{j})-h_{2}({v}_{i}))
+d_{i}^{\sigma}(h_{2}({v}_{\gamma})-h_{2}({v}_{i})).
\end{aligned}
\right. \label{MAS2}
\end{align}

The consensus objective for system \eqref{MAS2} is to guarantee the state errors of position and velocity both asymptotically converge to zero, i.e., $\lim_{t \rightarrow \infty} ||\omega_{i}||=0$ and $\lim_{t \rightarrow \infty} ||\nu_{i}||=0$ for $i=1,\ldots,N$.

In order to achieve the above control objective, the following assumption about $h_{1}(\cdot)$ and $h_{2}(\cdot)$ is used.

\begin{assumption} For $h_{1} \in \mathbb{R}[\cdot]: \mathbb{R}^n\rightarrow\mathbb{R}^n$ and $h_{2}\in \mathbb{R}[\cdot]: \mathbb{R}^n\rightarrow\mathbb{R}^n$, 
there exist a continuously differentiable, radially unbounded, positive definite polynomial function $V(\cdot):\mathbb{R}^n\rightarrow\mathbb{R}_{+}$, a vector valued function of polynomial $q(\cdot):\mathbb{R}^n\rightarrow\mathbb{R}^m$ with $q(\alpha)=0 \Leftrightarrow \alpha = 0$ such that for $\forall \alpha, \beta, \gamma \in \mathbb{R}^n$, there holds
\begin{align}
\nabla V (\alpha) h_{1}(\beta) &= \nabla V (\beta) h_{1}(\alpha), \label{A3-2} \\
 F_{\nabla V}(\alpha-\gamma,\beta-\gamma) & F_{h_{2}}(\alpha,\beta)  \geq  H_{q,\Psi}(\alpha-\gamma,\beta-\gamma). \label{A3-1}
\end{align}
\end{assumption}

\begin{remark}
For \eqref{A3-2}, it is not difficult to validate that the degree of $\nabla V(z)$ equals to the one of $h_{1}(z)$, i.e., if $V(\cdot) \in \mathbb{R}_{2m}[z]$, then $h_{1}(\cdot) \in \mathbb{R}_{2m-1}[z]$. For convenience, we can choose $h_{1}(\cdot)=[\nabla V(\cdot)]^{\top}$. In this case, $V(\cdot)$ is an even function due to the property of odd function $h_{1}(\cdot)$. And the condition \eqref{A3-1} is similar to the condition \eqref{A1}.
\end{remark}

\begin{theorem}
Assuming that Assumption 3 is valid. If the switching topology graph among the follower and leader agents is jointly connected across each time interval $[t_{\kappa},t_{\kappa+1})$ for $\kappa=0,1,\ldots$, then the consensus of system \eqref{MAS2} with second-order dynamics is achieved under protocol \eqref{Protocol2}.
\end{theorem}

\emph{Proof:} Let $V_{1}=\frac{1}{2}\sum_{i=1}^{N}\sum_{j \in \mathcal{N}_{i}}A_{ij}^{\sigma}V(\omega_{j}-\omega_{i})$, $V_{2}=\sum_{i=1}^{N}d_{i}^{\sigma}V(\omega_{i})$ and $V_{3}=\sum_{i=1}^{N}V(\nu_{i})$, where $V$ satisfies Assumption 3. Define $\bar{V}=V_{1}+V_{2}+V_{3}$. Taking the derivative of $\bar{V}$ yields
\begin{align}
\dot{V}_{1} &= \frac{1}{2}\sum_{i=1}^{N} \sum_{j \in \mathcal{N}_{i}} \nabla V(\omega_{j}-\omega_{i})A_{ij}^{\sigma}(h_{1}(\nu_{j})-h_{1}(\nu_{i}))  \nonumber\\
&= -\sum_{i=1}^{N} \sum_{j \in \mathcal{N}_{i}} \nabla V(\omega_{j}-\omega_{i})A_{ij}^{\sigma}h_{1}(\nu_{i}), \label{TH2-1}
\end{align}
and
\begin{align}
\dot{V}_{2} = \sum_{i=1}^{N}d_{i}^{\sigma} \nabla V(\omega_{i}) h_{1}({\nu}_{i}),  \label{TH2-2}
\end{align}
and
\begin{align}
\dot{V}_{3}     &= \sum_{i=1}^{N} \sum_{j \in \mathcal{N}_{i}} \nabla
V(\nu_{i})A_{ij}^{\sigma}h_{1}(\omega_{j}-\omega_{i}) \nonumber\\
                &\ \ \ - \sum_{i=1}^{N} d_{i}^{\sigma} \nabla V(\nu_{i}) h_{1}(\omega_{i})  \nonumber\\
                &\ \ \ + \sum_{i=1}^{N} \sum_{j \in \mathcal{N}_{i}} \nabla V(\nu_{i})A_{ij}^{\sigma}(h_{2}(v_{j})-h_{2}(v_{i})) \nonumber\\
                &\ \ \ + \sum_{i=1}^{N} d_{i}^{\sigma} \nabla V(\nu_{i}) (h_{2}(v_{\gamma})-h_{2}(v_{i})). \label{TH2-3}
\end{align}
According to \eqref{A3-2}, we have
\begin{align}
\dot{\bar{V}}   &=\dot{V}_{1} + \dot{V}_{2} + \dot{V}_{3}  \nonumber\\
                &=\sum_{i=1}^{N} \sum_{j \in \mathcal{N}_{i}} \nabla V(\nu_{i}) A_{ij}^{\sigma} (h_{2}(v_{j})-h_{2}(v_{i})) \nonumber\\
                &\ \ \ \ + \sum_{i=1}^{N} d_{i}^{\sigma} \nabla V(\nu_{i}) (h_{2}(v_{\gamma})-h_{2}(v_{i})) \nonumber\\
                &=-\sum_{i=1}^{N} \sum_{j =1}^{N} \nabla V(\nu_{i}) L_{ij}^{\sigma}h_{2}(v_{j}) \nonumber\\
                &\ \ \ \ - \sum_{i=1}^{N} d_{i}^{\sigma} \nabla V(\nu_{i}) (h_{2}(v_{i})-h_{2}(v_{\gamma})). \nonumber\\
                &= \sum_{i=1}^{N-1}\sum_{j=i+1}^{N} F_{\nabla V}(\nu_{i},\nu_{j}) L_{ij}^{\sigma} F_{h_{2}}(v_{i},v_{j}) \nonumber\\
                &\ \ \ \ -\sum_{i=1}^{N}d_{i}^{\sigma} \nabla V(\nu_{i})
                   F_{h_{2}}(v_{i},v_{\gamma}). \label{TH2-4}
\end{align}
It is seen from \eqref{A3-1} that
\begin{align}
\dot{\bar{V}} & \leq \sum_{i=1}^{N-1}\sum_{j=i+1}^{N} L_{ij}^{\sigma}
                  H_{q,\Psi}(\nu_{i},\nu_{j}) -\sum_{i=1}^{N} d_{i}^{\sigma} q^{\top}(\nu_{i}) \Psi q(\nu_{i}) \nonumber\\
                &= - \sum_{i=1}^{N} \sum_{j =1}^{N} L_{ij}^{\sigma} q^{\top}(\nu_{i}) \Psi q(\nu_{j}) -\sum_{i=1}^{N} d_{i}^{\sigma} q^{\top}(\nu_{i}) \Psi q(\nu_{i}) \nonumber\\
                &= -Q^{\top}(\nu)(G^{\sigma} \otimes \Psi)Q(\nu) \leq 0, \label{TH2-5}
\end{align}
where $Q(\nu)=[q^{\top}(\nu_{1}),\ldots,q^{\top}(\nu_{N})]^{\top}$. Similar to the proof of Theorem 1, we can prove that the velocity tracking errors of all the agents approach to zero. Furthermore, it follows from \eqref{MAS2} that the position tracking errors of all the agents also approach to zero, which indicates that the consensus of system \eqref{MAS2} under the jointly connected topology is achieved. The proof is thus completed.  \hfill{$\square$}

Afterwards, Assumption 4 and Proposition 2 are obtained as follows.

\begin{assumption} For $h_{1}(\cdot)\in \mathbb{R}[\cdot]:\mathbb{R}^n\rightarrow\mathbb{R}^n$ and $h_{2}(\cdot)\in \mathbb{R}[\cdot]:\mathbb{R}^n\rightarrow\mathbb{R}^n$, there exist an SOS polynomial function $V(\cdot):\mathbb{R}^n\rightarrow \mathbb{R}_{+}$, a vector valued function of polynomial $q(\cdot):\mathbb{R}^n\rightarrow\mathbb{R}^m$ with $q(\alpha)=0 \Leftrightarrow \alpha = 0$ such that for $\forall \alpha, \beta, \gamma \in \mathbb{R}^n$, there holds
\begin{align}
&q^{\top}(\alpha) \Psi q(\eta)- \alpha^{2m} \in {\rm SOS}[\alpha], \label{A4-1} \\
&\nabla V (\alpha) h_{1}(\beta)- \nabla V (\beta) h_{1}(\alpha)=0, \label{A4-2} \\
& F_{\nabla V}(\alpha-\gamma,\beta-\gamma)F_{h_{2}}(\alpha,\beta) \nonumber\\
&\quad \quad -H_{q,\Psi}(\alpha-\gamma,\beta-\gamma) \in {\rm SOS}[\alpha,\beta,\gamma]. \label{A4-3}
\end{align}
\end{assumption}

\begin{proposition}
Assuming that Assumption 4 is valid. If the switching topology graph among the follower and leader agents is jointly connected across each time interval $[t_{\kappa},t_{\kappa+1})$ for $\kappa=0,1,\ldots$, then the consensus of system \eqref{MAS2} with second-order dynamics is achieved under protocol \eqref{Protocol2}.
\end{proposition}

\section{Simulation}
\label{sim}
In this section, we present an example to demonstrate the effectiveness of our method.

Consider a second-order MAS consisting of a virtual leader (labeled by $0$) and four followers agents, each of which is a $2$-dimensional dynamical system. 
The switching graphs $\mathscr{G} \in \{\mathscr{G}_{1},\mathscr{G}_{2},\mathscr{G}_{3}\}$ for information communication among the leader and followers are shown in Fig. \ref{fig1}, and the switching frequency is $1000$ HZ for $\{\mathscr{G}_{1},\mathscr{G}_{2},\mathscr{G}_{3}\}$.
Note that each interaction graph may be not connected, while the union of graphs is connected, i.e., $\{\mathscr{G}_{1},\mathscr{G}_{2},\mathscr{G}_{3}\}$ is jointly connected.

Choose $\text{deg}(V)=4$, $\text{deg}(h_{1})=3$ and $\text{deg}(h_{2})=3$, where $\text{deg}(V)$ denotes the degree of $V(\cdot)$.
Let $\text{deg}(\hat{q})=2$ with
$\hat{q}=[x_{1}, x_{2}, x_{1}x_{2}, x_{1}^{2}, x_{2}^{2}]^{\top} \in \mathbb{R}^{5}$.
By solving the conditions \eqref{A4-1}, \eqref{A4-2} and \eqref{A4-3} under the MOSEK solver, the polynomial coupling functions are obtained for the distributed control protocol \eqref{Protocol2}, where  $h_{1}(z_{i})=[5.1181 z_{i1}^3 + 0.0737 z_{i1} z_{i2}^2 + 5.2307 z_{i1} - 0.0002 z_{i2}, 5.6633 z_{i2}^3 + 0.0737 z_{i1}^2 z_{i2} + 6.3900 z_{i2}- 0.0002 z_{i1}]^{\top}$ and 
$h_{2}(v_{i})=[9.5154 v_{i1}^3 + 1.7914 v_{i1} v_{i2}^2 + 6.0926 v_{i1}, 9.4734 v_{i2}^3 + 1.6446 v_{i1}^2 v_{i2} + 5.5437 v_{i2}]^{\top}$, 
and a quartic polynomial Lyapunov function is automatically synthesized to achieve consensus verification, where
$V(x_{1},x_{2})=1.2795 x_1^4 + 0.0369 x_1^2 x_2^2 + 1.4158 x_2^4 + 2.6153 x_1^2 - 0.0002 x_1 x_2 + 3.1950 x_2^2 + 1.8135$.
Note that if $\text{deg}(V)=2$, the quadratic Lyapunov functions cannot be found.

Given the initial conditions of the leader and followers as $z_{r}(0) = [1,1.5]^{\top}$, $z_{1}(0)=[0.1,1.0]^{\top}$, $z_{2}(0)=[0.4,0.5]^{\top}$, $z_{3}(0)=[0.5,1.2]^{\top}$, $z_{4}(0)=[0.8,2.0]^{\top}$, 
$v_{r}(0) = [0,0.5]^{\top}$, $v_{1}(0)=[0.1,0.2]^{\top}$, $v_{2}(0)=[0.3,0.4]^{\top}$, $v_{3}(0)=[0.5,0.6]^{\top}$, and $v_{4}(0)=[0.7,0.8]^{\top}$, the simulation results under the proposed protocol \eqref{Protocol2} show that both the velocity tracking error norm and the position tracking error norm of all the agents in the error system \eqref{MAS2} converge to zero (see Fig.~\ref{fig2}), confirming the theoretical and algorithmic results.

\begin{figure}[t]
\centering{\includegraphics[width=3.5in]{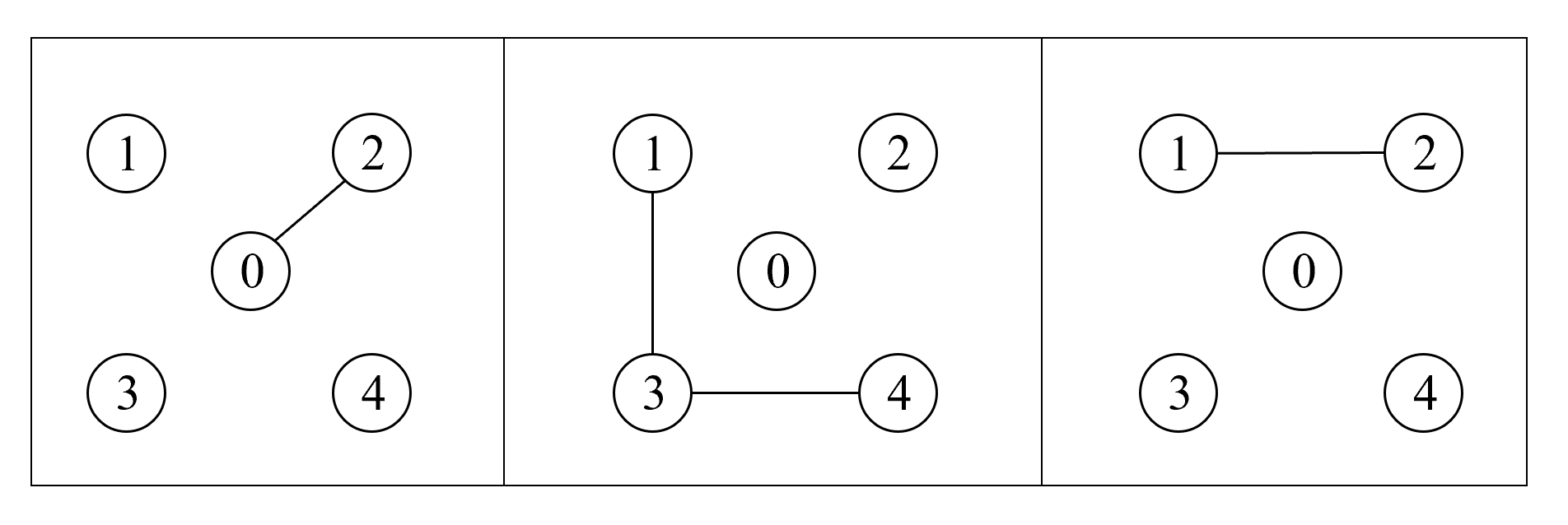}}
\caption{Switching interaction graphs.}
\label{fig1}
\end{figure}

\begin{figure}[t]
\centering{\includegraphics[width=3.2in]{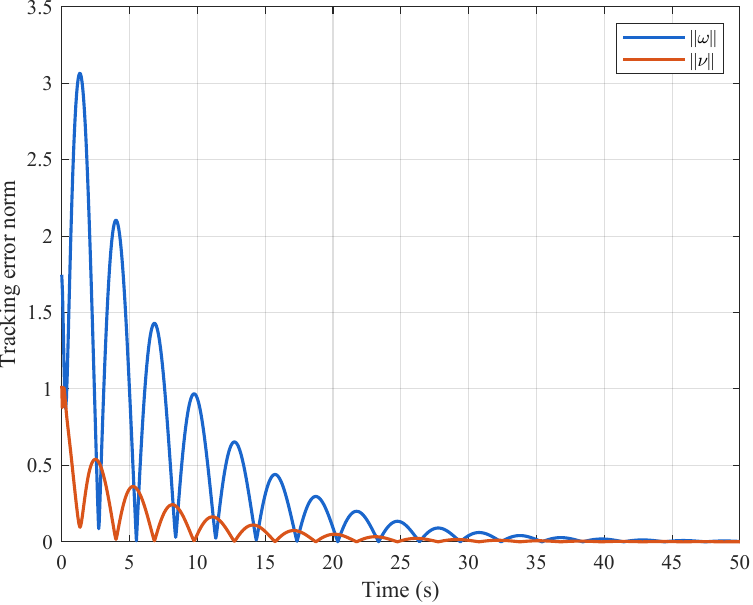}}
\caption{Time evolution of tracking error norm. Orange and blue curves correspond to velocity and position, respectively.}
\label{fig2}
\end{figure}

\section{Conclusion}
\label{4}
In this article, we investigated the consensus verification of MASs under the jointly connected topology subject to the distributed nonlinear control protocols.
Via the SOS decomposition for multivariate polynomials, we proposed a method to automatically synthesize the polynomial Lyapunov functions and polynomial coupling functions for our proposed protocols, to achieve consensus verification, regardless of the widely used quadratic Lyapunov functions might be non-existent.
Moreover, the proposed method was extended to address the second-order consensus of MASs.
Finally, a numerical example is given to demonstrate the performance of our approach.

In the future, building on our previous work \cite{Zhang:2022,zhang2024consensus,zhang2023stability,zhang2025automatic}, we will extend our approach to handle nonlinear MASs with homogeneous and heterogeneous dynamics under jointly connected topologies.

\bibliographystyle{ieeetr}
\bibliography{ref1}

\end{document}